\documentclass[12pt]{article}
\usepackage{fullpage}
\usepackage{amssymb,amsmath,amsfonts,amsthm,latexsym,enumerate,url,cases}
\numberwithin{equation}{section}
\usepackage{hyperref}
\hypersetup{colorlinks=true,citecolor=blue,linkcolor=blue,urlcolor=blue}

\newtheorem{tet}{Theorem}
\newtheorem*{Tet}{Theorem}
\newtheorem{lem}{Lemma}

\newtheorem{kov}[tet]{Corollary}
\theoremstyle{remark}

\newtheorem{prob}{Problem}
\theoremstyle{definition}

\newcommand{\seq}[2]{\left(#1\right)_{#2}}

\usepackage[left=50pt,right=50pt,top=30pt,bottom=80pt]{geometry}

\newcommand{\p}{\mathbb{P}}

\begin{document}

\title{Multiplicative complements I.}
\author{Anett Kocsis \thanks{E\" otv\" os Lor\' and University, Budapest, Hungary. Email: sakkboszi@gmail.com. Supported by the ÚNKP-21-1 New National Excellence Program of the Ministry for Innovation and Technology from the source of the National Research, Development and Innovation Fund. }, D\' avid Matolcsi \thanks{E\" otv\" os Lor\' and University, Budapest, Hungary. Email: matolcsidavid@gmail.com. Supported by the ÚNKP-21-1 New National Excellence Program of the Ministry for Innovation and Technology from the source of the National Research, Development and Innovation Fund.}, Csaba S\' andor \thanks{Department of Stochastics, Institute of Mathematics, Budapest University of
Technology and Economics, M\H{u}egyetem rkp. 3., H-1111, Budapest, Hungary. Department of Computer Science and Information Theory, Budapest University of Technology and Economics, M\H{u}egyetem rkp. 3., H-1111 Budapest, Hungary, MTA-BME Lend\"ulet Arithmetic Combinatorics Research Group,
  ELKH, M\H{u}egyetem rkp. 3., H-1111 Budapest, Hungary , MTA-BME Lend\"ulet Arithmetic Combinatorics Research Group H-1529 B.O. Box, Hungary. Email: csandor@math.bme.hu.
This author was supported by the NKFIH Grants No. K129335.} and Gy\"orgy T\H ot\H os \thanks{ Faculty of Mathematics and Computer Science, Babe\c{s}-Bolyai University. }}
\date{January 2022}

    \maketitle

\begin{abstract}
In this paper, we study how dense a multiplicative basis of order $h$ for $\mathbb{Z}^+\!$ can be, improving on earlier results. Upon introducing the notion of a \textit{multiplicative complement}, we present some tight density bounds.
\end{abstract}

\section{Introduction}
\def\MB{\mathop{\rm MB}\nolimits}%

Let $\mathbb{Z}^+\!$ denote the set of positive integers. For $A\subseteq\mathbb{Z}^+\!$ and $h\in\mathbb{Z}^+\!$, the \textit{multiplicative representation function} $S_{A,h}(n)$ \textit{of order $h$} counts the ordered representations of $n\in\mathbb{Z}^+\!$ as a product of $h$ elements of $A$; that is, we define
$$
S_{A,h}(n)=|\{(a_1,\ldots,a_h):a_i\in A\text{ and }a_1\cdots a_h=n\}|.
$$
We say that $A$ is a \textit{multiplicative basis of order $h$ for $\mathbb{Z}^+\mkern-1mu$} if all positive integers can be written in the form $a_1\cdots a_h$ for some $a_i\in A$, which we may also express as having $S_{A,h}(n)\ge 1$ for all $n\in \mathbb{Z}^+\!$. We denote the set of multiplicative bases like $A$ by $\MB_h$. The \textit{counting function} $A(n)$ now counts the elements of $A$ that are less than or equal to $n$; that is, we write $A(n)=|A\cap \{ 1,2,\ldots,n\}|$. It is easy to see that for $A\in\MB_h$, the prime numbers are necessarily members of $A$, and hence
$$
A(n)\ge\pi(n)=(1+o(1))\frac{n}{\log n}.
$$
In 1938, Raikov proved the following density bounds.

\begin{tet}[{\cite[Th. 1]{raikov1938multiplicative}}]
Let $h\in\mathbb{Z}^+\!$. Then:
\begin{enumerate}
\item For all $A\in\MB_h$, we have $\displaystyle\limsup_{x\to\infty}A(x)\frac{\log^{1-\frac{1}{h}}x}{x}\ge\frac{1}{\Gamma\left(\frac{1}{h}\right)}$.
\item There exists $A\in\MB_h$ with $\displaystyle\limsup_{x\to\infty}A(x)\frac{\log^{1-\frac{1}{h}}x}{x}<\infty$.
\end{enumerate}
\end{tet}

$\\$
Note that $\Gamma\left(\frac{1}{h}\right)=(1+o(1))h$ as $h\to \infty$. In 2018, Pach and S\' andor improved upon these inequalities.

\begin{tet}[{\cite[Th. 3]{pach2018multiplicative}}]
Let $h\in\mathbb{Z}^+\!$. Then:
\begin{enumerate}
\item For all $A\in\MB_h$, we have $\displaystyle\limsup_{x\to\infty}A(x)\frac{\log^{1-\frac{1}{h}}x}{x}\ge\frac{\sqrt{6}}{e\pi}$.
\item There exists $C>0$ such that for each $h\ge 2$, one can find $A\in\MB_h$ with
\[
\mkern-100mu\limsup_{x\to \infty}A(x)\frac{\log^{1-\frac{1}{h}}x}{x}=C.
\]
\end{enumerate}
\end{tet}

\vskip16pt
The following theorem provides an even better lower bound.

\begin{tet}\label{raikov:better}
Let $h\in\mathbb{Z}^+\!$. For all $A\in\MB_h$, we have $\displaystyle\limsup_{x\to \infty}A(x)\frac{\log^{1-\frac{1}{h}}x}{x}\ge\frac{\sqrt[h]{h!}}{\Gamma\left(\frac{1}{h}\right)}$.
\end{tet}

Let us note that for $h\ge 2$, the sequence $\frac{\sqrt[h]{h!}}{\Gamma\left(\frac{1}{h}\right)}$ is decreasing in $h$, and its limit is $\frac{1}{e}$. Consequently:

\begin{kov}
With $h\ge 2$, for all $A\in\MB_h$, we have $\displaystyle\limsup_{x\to\infty}A(x)\frac{\log^{1-\frac{1}{h}}x}{x}>\frac{1}{e}$.
\end{kov}

\vskip16pt
For $A_1,\ldots,A_h\subseteq\mathbb{Z}^+\!$, we define the \textit{common multiplicative representation function} as
$$
S_{A_1,\ldots,A_h}(n)=|\{(a_1,\ldots,a_h):a_i\in A_i\text{ and }a_1\cdots a_h=n\}|
$$
with $n\in\mathbb{Z}^+$. Raikov's theorem may now be generalised to these functions as follows.

\begin{tet}\label{mult:sup}
Let $A_1,\ldots,A_h\subseteq\mathbb{Z}^+\!$. Pick $\tau_1,\ldots,\tau _h\in(0,1)$ such that $\displaystyle\sum_{i=1}^h\tau_i=1$, and assume that
$$
\advance\abovedisplayskip by-8pt
\limsup_{x\to\infty}A_i(x)\frac{\log^{1-\tau_i}x}{x}<\infty
$$
for $1\le i \le h$. Then, we have
\[
\left(\liminf_{x\to\infty}\frac{\sum_{n\le x}S_{A_1,\ldots,A_h}(n)}{x}\right)\;\prod_{i=1}^{h}\;\frac{1}{\Gamma(\tau_i)}\;\le\;\prod_{i=1}^{h}\;\limsup_{x\to\infty}A_i(x)\frac{\log^{1-\tau_i}x}{x}.
\]
\end{tet}

\vskip16pt
\def\MC{\mathop{\rm MC}\nolimits}%
Let $A_i\subseteq \mathbb{Z}^+$ for $1\le i\le h$. We shall refer to the $h$-tuple $(A_1,\ldots,A_h)$ as \textit{a multiplicative complement of order $h$} if all positive integers can be written in the form $a_1\cdots a_h$ with $a_i\in A_i$, which we may also express as having $S_{A_1,\ldots,A_h}(n)\ge 1$ for all $n\in \mathbb{Z}^+$. We denote the set of multiplicative complements of order $h$ by $\MC_h$. The following is then a direct consequence of the above theorem.

\begin{kov}
Let $h\in\mathbb{Z}^+\!$ and $(A_1,\ldots, A_h)\in\MC_h$. Pick $\tau_1,\ldots,\tau_h\in(0,1)$ such that $\displaystyle\sum_{i=1}^h\tau_i=1$ and
$$
\advance\abovedisplayskip by-8pt
\limsup_{x\to\infty}A_i(x)\frac{\log^{1-\tau_i}x}{x}<\infty
$$
for $1\le i \le h$. Then, we have
\[
\prod_{i=1}^{h}\;\frac{1}{\Gamma(\tau_i)}\;\le\;\prod_{i=1}^{h}\;\limsup_{x\to\infty}A_i(x)\frac{\log^{1-\tau_i}x}{x}.
\]
\end{kov}

\vskip16pt
This inequality is sharp.

\begin{tet}\label{precision:bound:2}
Let $h\in\mathbb{Z}^+\!$. Pick $\tau_1,\ldots,\tau_h\in(0,1)$, $a_1,\ldots,a_h\in(0,\infty)$ with $\displaystyle\sum_{i=1}^h\tau_i=1$ and $\displaystyle\prod_{i=1}^ha_i=1$. Then, there exists $(A_1,\ldots,A_h)\in\MC_h$ such that
\[
\lim_{x\to\infty} A_i(x)\frac{\log^{1-\tau_i}x}{x}=\frac{a_i}{\Gamma(\tau_i)}
\]
for all $1\le i\le h$.
\end{tet}

\vskip16pt
A simple corollary to these theorems may now be formulated.

\begin{kov}\label{preicision:mult:comp}
Let $h\in\mathbb{Z}^+\!$. Then:
\begin{enumerate}
\item For all $(A_1,\ldots,A_h)\in\MC_h$, we have $\displaystyle\limsup_{x\to\infty}\max\{A_1(x),\ldots,A_h(x)\}\frac{\log^{1-\frac{1}{h}} x}{x}\ge\frac{1}{\Gamma\left(\frac{1}{h}\right)}$.{\parfillskip=43pt\par}
\item There is $(A_1,\ldots,A_h)\in\MC_h$ for which $\displaystyle\lim_{x\to\infty}\max\{A_1(x),\ldots,A_h(x)\}\frac{\log^{1-\frac{1}{h}}x}{x}=\frac{1}{\Gamma\left(\frac{1}{h}\right)}$.
\end{enumerate}
\end{kov}

\vskip16pt
It is easy to see that if $(A_1,\dots A_h)\in\MC_h$, then we necessarily have $A_1\cup\cdots\cup A_h\in\MB_h$, and hence:

\begin{kov}
With $h\in\mathbb{Z}^+$, there exists $A\in\MB_h$ such that
\[
\limsup_{x\to\infty}A(x)\frac{\log^{1-\frac{1}{h}}x}{x}=\frac{h}{\Gamma\left(\frac{1}{h}\right)}.
\]
\end{kov}

\vskip16pt
One can readily check that $\frac{h}{\Gamma(\frac{1}{h})}=\frac{1}{\Gamma(1+\frac{1}{h})}$ and $\displaystyle\min_{h\ge 2}\Gamma\Big(1+\frac{1}{h}\Big)=\frac{\sqrt{\pi}}{2}$. This yields the next statement.
\begin{kov}

With $h\ge 2$, there exists $A\in\MB_h$ such that
\[
\limsup_{x\to\infty}A(x)\frac{\log^{1-\frac{1}{h}}x}{x}=\frac{2}{\sqrt{\pi}}.
\]
\end{kov}

\vskip16pt
In 2018, Pach and S\' andor also proved density bounds for the limit inferior.

\begin{tet}[{\cite[Th. 4]{pach2018multiplicative}}]\label{PS}
Let $h\in\mathbb{Z}^+\!$. Then:
\begin{enumerate}
\item For all $A\in\MB_h$, we have $\displaystyle\liminf\limits_{x\to \infty}\frac{A(x)}{\frac{x}{\log x}}>1$.
\item For all $\varepsilon>0$, there exists $A\in\MB_h$ with $\displaystyle\liminf_{x\to\infty}\frac{A(x)}{\frac{x}{\log x}}\le 1+\varepsilon$.
\end{enumerate}
\end{tet}

\vskip16pt
For multiplicative complements, we may formulate the analogue as follows.

\begin{tet}\label{liminf:max}
Let $h\in\mathbb{Z}^+\!$. Then:
\begin{enumerate}
\item For all $(A_1, A_2,\ldots,A_h)\in\MC_h$, we have $\displaystyle\liminf\limits_{x\to\infty} \frac{\max\{A_1(x),\ldots,A_h(x)\}}{\frac{x}{\log x}}>\frac{1}{h}$.
\item For all $\varepsilon>0$, there exists $(A_1,\ldots,A_h)\in\MC_h$ with $\displaystyle\liminf_{x\to\infty}\frac{\max\{A_1(x),\ldots,A_h(x)\}}{\frac{x}{\log x}}\le\frac{1}{h}+\varepsilon$.
\end{enumerate}
\end{tet}

\vskip16pt
Let us, finally, propose some problems for further research. Note that for $h\ge 2$, we have
$$
\mkern14mu\frac{1}{e}<\frac{\sqrt[h]{h!}}{h\Gamma\left(1+\frac{1}{h}\right)}\le\inf_{A\in\MB_h}\bigg\{\limsup_{x\to\infty}\frac{A(x)}{\frac{x}{\log^{1-\frac{1}{h}}x}}\bigg\}\le\frac{1}{\Gamma\left(1+\frac{1}{h}\right)}\le\frac{2}{\sqrt{\pi}}.
$$

\begin{prob}\label{prob1}
Is it true that for all $h\ge 2$ and $A\in\MB_h$, we have
$$
\limsup_{x\to\infty}A(x)\frac{\log^{1-\frac{1}{h}}x}{x}\ge\frac{1}{\Gamma\left(1+\frac{1}{h}\right)}\,?
$$
\end{prob}

As answering this question seems to be hard, we shall simplify it. Note that for $A\in\MB_2$, we have
$$
\frac{\sqrt{2}}{\sqrt{\pi}}\le\limsup_{x\to \infty}\frac{A(x)}{\frac{x}{\sqrt{\log x}}}.
$$

\begin{prob}\label{prob2}
Is it true that there exists $\delta >0$ such that for all $A\in\MB_2$, we have
$$
\frac{\sqrt{2}}{\sqrt{\pi }}+\delta\le\limsup_{x\to\infty}\frac{A(x)}{\frac{x}{\sqrt{\log x}}}\,?
$$
\end{prob}\label{prob3}

Picking now $A\in\MB_2$, obviously, $S_{A,2}(n)\ge 2$ holds as long as $n\in\mathbb{Z}^+\!$ is not a perfect square, and so $\displaystyle\sum_{n\le x}S_{A,2}(n)\ge x-\lfloor\sqrt{x}\rfloor$. In particular,
$$
\advance\abovedisplayskip by-16pt
\liminf _{x\to \infty}\frac{\sum_{n\le x}S_{A,2}(n)}{x}\ge 2.
$$
Thus, according to Theorem \ref{mult:sup}, we can further reduce Problem \ref{prob2} to the following.

\begin{prob}
Is it true that there exists $\delta _0>0$ such that for all $A\in\MB_2$, we have
$$
\liminf_{x\to\infty}\frac{\sum_{n\le x}S_{A,2}(n)}{x}\ge 2+\delta _0\,?
$$
\end{prob}

Concluding, note how Theorem \ref{mult:sup} claims $\liminf$ to be finite. As for $\limsup$,
the analogue is:

\begin{prob}\label{prob4}
Let $h\in\mathbb{Z}^+$ and $(A_1,\ldots,A_h)\in\MC_h$. Pick $\tau_1,\ldots,\tau_h\in(0,1)$ such that $\displaystyle \sum_{i=1}^h\tau_i=1$ and
$$
\advance\abovedisplayskip by-8pt
\limsup_{x\to\infty}A_i(x)\frac{\log^{1-\tau_i}x}{x}<\infty
$$
for $1\le i \le h$. Is it true that
\begin{equation}
\limsup_{x\to \infty}\frac{\sum_{n\le x}S_{A_1,\dots ,A_h}(n)}{x}<\infty\,?
\end{equation}
\end{prob}

\section{Proofs}

In what follows, for $A\subseteq \mathbb{Z}^+$, we write $A[s]= \sum_{a\in A}\frac{1}{a^s}$, where $s>1$.

The proof of Theorem \ref{raikov:better} and that of Theorem \ref{mult:sup} are based on the next lemma, coming from Raikov's paper yet not explicitly stated there.

\begin{lem}\label{lem:0:infsup}
Choose $A\subseteq\mathbb{Z}^+\!$ and $\tau \in(0,1)$ such that $\displaystyle\limsup_{x\to\infty}A(x)\frac{\log^{1-\tau }x}{x}<\infty$. Then,
\[
\limsup_{s\searrow1}(s-1)^{\tau}A[s]\le\Gamma(\tau)\limsup_{x\to\infty}A(x)\frac{\log^{1-\tau}x}{x}.
\]
\end{lem}

\begin{proof}
We know that for all $s>1$, we have
\[
s\int_{1}^{\infty}\frac{A(x)}{x^{s+1}}dx=\sum_{a\in A}s\int_a^{\infty}\frac{1}{x^{s+1}}dx=\sum_{a\in A}\frac{1}{a^s}=A[s].
\]
Pick now $a>\alpha$, where
\[
\mkern-4mu\alpha=\limsup_{x\to\infty}A(x)\frac{\log^{1-\tau}x}{x}.
\]
Then, there exists $x_0>1$ such that for all $x\ge x_0$, we have
\[
\frac{A(x)}{x}\le\frac{a}{\log^{1-\tau}x}.
\]
Consequently,
\[
A[s]\le s\int_{1}^{x_0}\frac{A(x)}{x^{s+1}}dx+sa\int_{x_0}^{\infty}\frac{1}{x^s\log^{1-\tau}x}dx.
\advance\belowdisplayshortskip by8pt
\]
Changing variable and writing $t = (s-1)\log x$ in the rightmost integral, we get
\[
A[s]\le s\int_{1}^{x_0}\frac{A(x)}{x^{s+1}}dx+\frac{s}{(s-1)^{\tau}}a\int_{(s-1)\log x_0}^{\infty}\!\!\!\!t^{\tau-1}e^{-t}dt.
\]
Introducing the gamma function, we may now rewrite the previous inequality as
\[
(s-1)^{\tau}A[s]\le s(s-1)^{\tau}\!\!\int_{1}^{x_0}\frac{A(x)}{x^{s+1}}dx - sa\!\!\int_{0}^{(s-1)\log x_0}\!\!\!t^{\tau-1}e^{-t}dt+\Gamma(\tau)sa.
\]
Since $s\searrow1$, we have
\[
s(s-1)^{\tau}\!\!\int_{1}^{x_0}\frac{A(x)}{x^{s+1}}dx\to 0,\quad sa\!\!\int_{0}^{(s-1)\log x_0}\!\!\!t^{\tau-1}e^{-t}dt\to 0\;\text{ and }\;\Gamma(\tau)sa\to \Gamma(\tau)a,
\]
and thus
\[
\limsup_{s\searrow1}(s-1)^{\tau}A[s]\le\Gamma(\tau)a.
\]
This proves the desired inequality.
\end{proof}

The following lemma improves on Raikov's result.

\begin{lem}\label{factorial_raikov}
For a multiplicative basis $A$ of order $h$ for $\mathbb{Z}^+$, we have
\[
\limsup_{s\searrow1}(s-1)^{\frac{1}{h}}A[s]\ge\sqrt[h]{h!}\,.
\]
\end{lem}

\begin{proof}
Without loss of generality, we may assume that
\[
\mkern20mu\limsup_{s\searrow1}(s-1)^{\frac{1}{h}}A[s]=\alpha<\infty.
\]
Set $B=A\mkern1mu\cup\mkern1mu\{a^2\mkern1mu:\mkern1mu a\in A\}$. Then,
\[
\frac{1}{h!}A[s]^{h}+B[s]^{h-1}=\sum_{n\ge 1}\frac{1}{n^s}\left(\frac{S_{A,h}(n)}{h!}+S_{B,h-1}(n)\right).
\]
Given that $A$ is a multiplicative basis of order $h$ for $\mathbb{Z}^+\!$, we can write $n=a_1\cdots a_h$ for some $a_i\in A$. If $S_{B,h-1}(n)=0$, then $a_i\neq a_j$ for all $i\,\neq\, j$, meaning that $S_{A,h}(n)\ge h!$. That is, for all $n\ge 1$, we have
\[
\frac{S_{A,h}}{h!}+S_{B,h-1}(n)\ge 1.
\]
Consequently, for all $s>1$, we get
\[
\frac{1}{h!}A[s]^{h}+B[s]^{h-1}\ge\zeta(s).
\]
Note that $B[s]\le A[s]+A[2s]$ for any $s>1$, and hence
\[
(s-1)B[s]^{h-1}\le (s-1)A[2s]^{h-1}+\sum_{j=1}^{h-1}\binom{h-1}{j}\left((s-1)^{\frac{1}{j}}A[s]\right)^jA[2s]^{h-1-j}.
\]
Clearly, $A[2s]\to A[2]\le\zeta(2)<\infty$ as $s\searrow 1$. Since $\alpha<\infty$, for all $1\le j\le h-1$, we get
\[
\lim_{s\searrow1}(s-1)^{\frac{1}{j}}A[s] = 0,
\]
and so
\[
\lim_{s\searrow1}(s-1)B[s]^{h-1} = 0.
\]
Therefore, we can conclude that
\[
\limsup_{s\searrow1}(s-1)\frac{1}{h!}A[s]^{h}\ge\limsup_{s\searrow1}(s-1)\zeta(s)= 1,
\]
which then completes the proof.
\end{proof}

\begin{proof}[Proof of theorem \ref{raikov:better}]
Using Lemma \ref{lem:0:infsup} and Lemma \ref{factorial_raikov}, for a multiplicative basis $A$ of order $h$, we have
\[
\limsup_{x\to\infty}A(x)\frac{\log^{1-\frac{1}{h}}}{x}\ge\frac{1}{\Gamma(\frac{1}{h})}\limsup_{s\searrow 1}(s-1)^{\frac{1}{h}}A[s]\ge\frac{\sqrt[h]{h!}}{\Gamma(\frac{1}{h})}.
\]
\end{proof}

\vskip8pt
\begin{proof}[Proof of theorem \ref{mult:sup}]
We prove by contradiction. Suppose that
$$
\liminf_{x\to \infty}\frac{\sum_{n\le x}S_{A_1,\dots ,A_h}(n)}{x}>\prod_{i=1}^{h}\left(\Gamma(\tau_i)\limsup_{x\to\infty}A_i(x)\frac{\log^{1-\tau_i}x}{x}\right).
$$
Pick $G$ satisfying
$$
\prod_{i=1}^{h}\left(\Gamma(\tau_i)\limsup_{x\to\infty}A_i(x)\frac{\log^{1-\tau_i}x}{x}\right) <G<\liminf_{x\to \infty}\frac{\sum_{n\le x}S_{A_1,\ldots,A_h}(n)}{x}.
$$
We can find $n_0>0$ such that for all $n\ge n_0$, we have
\[
b_n = \sum_{k\le n}S_{A_1,\dots  ,A_h}(k)\ge G\cdot n.
\]
Pick $n> n_0$ and $s>1$. Then,
\begin{align*}
\sum_{k=1}^{n}\frac{S_{A_1,\ldots,A_h}(k)}{k^s}&=\sum_{k=1}^{n}\frac{b_{k}-b_{k-1}}{k^s}=\frac{b_n}{n^s}+\sum_{k=1}^{n-1}\frac{b_k}{k}\left(\frac{k}{k^s}-\frac{k}{(k+1)^s}\right)\\
&\ge G\mkern-4mu\sum_{n_0\le k< n}\mkern-4mu\left(\frac{1}{k^{s-1}}-\frac{1}{(k+1)^{s-1}}+\frac{1}{(k+1)^s}\right)\ge G\mkern-8mu\sum_{n_0< k\le n}\frac{1}{k^s}.
\end{align*}
It follows that
\[
\sum_{k=1}^{\infty}\frac{s_{A_1,\dots ,A_h}(k)}{k^s}\ge G\cdot\zeta(s)-G\sum_{k\le n_0}\frac{1}{k^s}.
\]
Note that $\displaystyle G\sum_{k\le n_0}\frac{1}{k^s}$ tends to $G\sum_{k\le n_0}\frac{1}{k}$ as $s\searrow 1$, which expression is finite, and thus
\[
\liminf_{s\searrow1}(s-1)\sum_{k=1}^{n}\frac{S_{A_1,\ldots,A_h}(k)}{k^s}\ge\liminf_{s\searrow 1}(s-1)G\zeta(s)=G.
\]
From
\[
\mkern-18mu\sum_{k=1}^{n}\frac{S_{A_1,l\dots,A_h}(k)}{k^s}=\prod_{i=1}^{h}A_i[s]
\]
and using $\displaystyle \sum_{i=1}^h\tau _i=1$, we get
\[
\advance\abovedisplayskip by-8pt
G\le\liminf_{s\searrow1}(s-1)\prod_{i=1}^{h}A_i[s]=\liminf_{s\searrow1}\prod_{i=1}^{h}\left( (s-1)^{\tau _i}A_i[s]\right) \le \prod_{i=1}^{h}\limsup_{s\searrow1}(s-1)^{\tau_i}A_i[s].
\]
Using now Lemma \ref{lem:0:infsup}, we get
\[
\limsup_{s\searrow1}(s-1)^{\tau _i}A_i[s]\le\Gamma(\tau_i)\limsup_{x\to\infty}A_i(x)\frac{\log^{1-\tau_i}x}{x}.
\]
Putting it all together, it follows that
\[
G\le \prod_{i=1}^{h}\left( \Gamma(\tau _i)\limsup_{x\to\infty}A_i(x)\frac{\log^{1-\tau_i}x}{x}\right)\mkern-2mu,
\]
which, however, contradicts the choice of $G$.
\end{proof}	

\vskip8pt
\begin{proof}[Proof of Theorem \ref{precision:bound:2}]
We prove by showing that the set $P$ of prime numbers may be written as
\begin{equation}\label{part1}
\mkern-10mu P=P_1\cup \dots \cup P_h
\end{equation}
with $P_i\cap P_j=\emptyset $ for $i\ne j$ and with the partitions $P_i$ subject to
\begin{equation}\label{part2}
P_i(x)=\tau_i\frac{x}{\log x}+O\Big(\frac{x}{\log^2x}\Big)
\end{equation}
and
\begin{equation}\label{part3}
\lim_{x\to\infty}\bigg(\prod_{{p\le x,\;p\in P_i}}\frac{p}{p-1}\bigg)\bigg(\prod_{\substack{p\le x}}\frac{p-1}{p}\bigg)^{\tau_i}=a_i.
\advance\belowdisplayshortskip by8pt
\end{equation}
In order to see why the existence of such a partition indeed implies the theorem, we shall invoke the Wirsing--Odoni theorem (see \cite[Pr. 4]{finch2010roots}).
\begin{Tet}\textbf{(Wirsing--Odoni)}
Let $f$ be a multiplicative function. Assume that there exist constants $u$, $v$ such that $0 \le f(p^k)<uk^v$ for all primes $p$ and all positive integers $k$. Assume further that there exist real numbers $\xi>0$ and $1<r<2$ such that
$$\sum\limits_{x\ge p\text{ prime}}f(p)=\xi\frac{x}{\log x}+O\Big(\frac{x}{\log^r x}\Big)$$
as $x\to \infty$. Then, the product
$$
C_f=\frac{1}{\Gamma(\xi)}\prod_{p\text{ prime}}\Big(1+\frac{f(p)}{p}+\frac{f(p^2)}{p^2}+\frac{f(p^3)}{p^3}+\cdots\Big)\Big(1-\frac{1}{p}\Big)^{\xi}
$$
over the primes is convergent (and positive), and
$$
\sum\limits_{n\le x}f(n)=C_f\frac{x}{\log^{1-\xi}x}+O\Big(\frac{x}{\log^{r-\xi}x}\Big)
$$
as $x\to\infty$.
\end{Tet}

Back to the proof, introduce
$$
A_i=\{n:\textrm{each prime factor of $n$ belongs to the set $P_i$}\}
$$
for $1\le i\le h$. We wish to apply the Wirsing--Odoni theorem to the multiplicative function defined as $f(p^k)=1$ if $p\in P_i$ and $f(p^k)=0$ if $p\notin P_i$ with $p$ a prime. It follows that
$$
\sum_{p\le x}f(p)=P_i(x)=\tau _i\frac{x}{\log x}+O\Big(\frac{x}{\log ^2x}\Big),
$$
and by condition (\ref{part3}), also
$$C_f=\frac{1}{\Gamma(\tau _i)}\prod_{p\textrm{ prime}}\Big(1+\frac{f(p)}{p}+\frac{f(p^2)}{p^2}+\frac{f(p^3)}{p^3}+\ldots\Big)\Big(1-\frac{1}{p}\Big)^{\tau_i}=\frac{a_i}{\Gamma(\tau _i)}.
$$
According to the Wirsing--Odoni theorem, we have
$$
\sum_{n\le x}f(n)=A_i(x)=\Big(\frac{a_i}{\Gamma(\tau _i)}+o(1)\Big)\frac{x}{\log^{1-\tau_i}x}.
$$
As $(A_1,\ldots,A_h)\mkern-1mu\in\mkern-1mu\MC_h$, we win.

The existence of a suitable partition subject to conditions (\ref{part1}), (\ref{part2}) and (\ref{part3}) is shown via the following lemma.

\begin{lem}\label{prim_particio}
Let $P=\{ p_1,p_2,\dots \}$ be the set of prime numbers with $p_1<p_2<\cdots$, and pick $0<\kappa\le 1$. Consider now $Q\subseteq P$ subject to
$$
Q(x)=\kappa\frac{x}{\log x}+O\Big(\frac{x}{\log^2x}\Big)
$$
and such that there exists $K\in\mathbb{Z}^+\!$ with
$$
Q\cap \{p_n,\ldots,p_{n+K}\}\ne\emptyset
$$
for all $n\in\mathbb{Z}^+\!$. Let $0<\tau <\kappa$, $a\in \mathbb{R}^+$. Then, there exist $R\subseteq Q$ and $L\in\mathbb{Z}^+\!$ such that for all $n\in \mathbb{Z}^+\!$,
$$
R\cap \{p_n+,\ldots,p_{n+L}\}\ne\emptyset,\quad R(x)=\tau\frac{x}{\log x}+O\Big(\frac{x}{\log^2x}\Big)
$$
and
$$
\lim_{x\to\infty}\bigg(\prod_{p\le x,\;p\in R}\frac{p}{p-1}\bigg)\bigg(\prod_{p\le x}\frac{p-1}{p}\bigg)^{\tau}=a.
$$
\end{lem}

\vskip16pt
The desired partition of the prime numbers now follows recursively. Indeed, it is known that
$$
P(x)=\frac{x}{\log x}+O\Big(\frac{x}{\log^2 x}\Big).
$$
Then, by Lemma \ref{prim_particio}, there exist $Q_1\subseteq Q_0=P$ and $L_1\in\mathbb{Z}^+\!$ such that for all $n\in \mathbb{Z}^+\!$, we have
$$
Q_1\cap\{p_n,\ldots,p_{n+L_1}\}\ne\emptyset,\quad Q_1(x)=(\tau_2+\cdots+\tau_h)\frac{x}{\log x}+O\Big(\frac{x}{\log^2x}\Big)
$$
and
$$
\lim_{x\to\infty}\bigg(\prod_{p\le x,\;p\in Q_1}\frac{p}{p-1}\bigg)\bigg(\prod_{p\le x}\frac{p-1}{p}\bigg)^{\tau_2+\dots+\tau_h}=a_2\cdots a_h.
\advance\belowdisplayshortskip by8pt
$$
Set $P_1=Q_0\setminus Q_1$. As $\tau _1+\dots +\tau _h=1$ and $a_1\cdots a_h=1$, we get
$$
P_1(x)=\tau_1\frac{x}{\log x}+O\Big(\frac{x}{\log^2x}\Big)
$$
and
$$
\lim_{x\to\infty}\bigg(\prod_{p\le x,\;p\in P_1}\frac{p}{p-1}\bigg)\bigg(\prod_{p\le x}\frac{p-1}{p}\bigg)^{\tau_1}=a_1.
\advance\belowdisplayshortskip by8pt
$$
Continuing in a similar fashion, say that we have already defined $P_1,\dots ,P_j\subseteq P$ for some $1\le j<h$ with $P_u\cap P_v=\emptyset $ for $1\le u<v\le j$ and with the partitions $P_i$ subject to
$$
P_i(x)=\tau_i\frac{x}{\log x}+O\Big(\frac{x}{\log ^2x}\Big)
$$
and
$$
\lim_{x\to\infty}\bigg(\prod_{p\le x,\;p\in P_i}\frac{p}{p-1}\bigg)\bigg(\prod_{p\le x}\frac{p-1}{p}\bigg)^{\tau _i}=a_i
\advance\belowdisplayshortskip by8pt
$$
for $1\le i\le j$. Assume further that for
$$
\advance\abovedisplayskip by-4pt
Q_j=P\setminus (P_1\cup\cdots\cup P_j),
$$
there is a positive integer $L_j$ such that
$$
Q_j\cap\{p_{n+1},\ldots,p_{n+L_j}\}\ne\emptyset
\advance\belowdisplayskip by-8pt
$$
for all $n\in\mathbb{Z}^+$. Then,
$$
(P_1\cup\cdots\cup P_j)(x)=(\tau_1+\cdots+\tau_j)\frac{x}{\log x}+O\Big(\frac{x}{\log^2x}\Big)
$$
and
$$
\lim_{x\to\infty}\bigg(\prod_{p\le x,\;p\in P_1\cup\cdots\cup P_j}\frac{p}{p-1}\bigg)\bigg(\prod_{p\le x}\frac{p-1}{p}\bigg)^{\tau_1+\cdots+\tau_j}=a_1\cdots a_j.
\advance\belowdisplayshortskip by8pt
$$
If $j\le h-2$, then there exist $Q_{j+1}\subseteq Q_j$ and $L_{j+1}\in\mathbb{Z}^+\!$ such that for all $n\in\mathbb{Z}^+\!$,
$$Q_{j+1}\cap\{p_n,\ldots,p_{n+L_{j+1}}\}\ne\emptyset,\quad Q_{j+1}=(\tau_{j+2}+\cdots+\tau_h)\frac{x}{\log x}+O\Big(\frac{x}{\log^2x}\Big)$$
and
$$
\lim_{x\to\infty}\bigg(\prod_{p\le x,\;p\in Q_{j+1}}\frac{p}{p-1}\bigg)\bigg(\prod_{p\le x}\frac{p-1}{p}\bigg)^{\tau_{j+2}+\cdots+\tau_h}=a_{j+2}\cdots a_h.
\advance\belowdisplayshortskip by8pt
$$
As $\tau_1+\dots+\tau_h=1$ and $a_1\cdots a_h=1$, we get
$$
P_{j+1}(x)=\tau_{j+1}\frac{x}{\log x}+O\Big(\frac{x}{\log ^2x}\Big)
$$
and
$$
\lim_{x\to\infty}\bigg(\prod_{p\le x,\;p\in P_{j+1}}\frac{p}{p-1}\bigg)\bigg(\prod_{p\le x}\frac{p-1}{p}\bigg)^{\tau_{j+1}}=a_{j+1}.
\advance\belowdisplayshortskip by8pt
$$
If $j=h-1$, set $P_h=Q_{h-1}=P\setminus(P_1\cup\cdots\cup P_{h-1})$. Note that we have
$$
(P_1\cup\cdots\cup P_{h-1})(x)=(\tau_1+\cdots+\tau_{h-1})\frac{x}{\log x}+O\Big(\frac{x}{\log^2x}\Big)
$$
and
$$
\lim_{x\to\infty}\bigg(\prod_{p\le x,\;p\in P_1\cup\cdots\cup P_{h-1}}\frac{p}{p-1}\bigg)\bigg(\prod_{p\le x}\frac{p-1}{p}\bigg)^{\tau_1+\cdots+\tau_{h-1}}=a_1\cdots a_{h-1}.
\advance\belowdisplayshortskip by8pt
$$
Once again, as $\tau _1+\cdots+\tau_h=1$ and $a_1\cdots a_h=1$, we get
$$
P_h(x)=\tau_h\frac{x}{\log x}+O\Big(\frac{x}{\log^2x}\Big)
$$
and
$$
\lim_{x\to\infty}\bigg(\prod_{p\le x,\;p\in P_h}\frac{p}{p-1}\bigg)\bigg(\prod_{p\le x}\frac{p-1}{p}\bigg)^{\tau_h}=a_h,
\advance\belowdisplayshortskip by8pt
$$
which then completes the argument.

It hence remains to prove Lemma \ref{prim_particio}.

\begin{proof}[Proof of Lemma \ref{prim_particio}]
Let us construct $R\subseteq Q$ by picking the elements of $Q$ and adding $p\in Q$ to $R$ if $R(p-1)<\frac{\tau}{\kappa}Q(p)$. This then yields
$$
R(x)=\frac{\tau}{\kappa}Q(x)+O(1)
$$
and so
\begin{equation}\label{Rbecsles}
R(x)=\tau\frac{x}{\log x}+O\Big(\frac{x}{\log^2 x}\Big).
\end{equation}
Consequently, with $N$ large enough,
\begin{equation}\label{tartalmaz1}
|R\cap\{p_n,p_{n+1},\ldots,p_{n+N-1}\}|\ge 2,
\end{equation}
\begin{equation}\label{tartalmaz2}
|(Q\setminus R)\cap\{p_n,p_{n+1},\ldots,p_{n+N-1}\}|\ge 2
\end{equation}
and there exist numbers $r\in R\cap\{p_n,p_{n+1},\ldots,p_{n+N-1}\}$, $s\in(Q\setminus R)\cap\{p_n,p_{n+1},\ldots,p_{n+N-1}\}$ such that $r<s$ for all $n\in\mathbb{Z}^+\!$. We can also find sequences $(r_k)_k$, $(s_k)_k$ of primes with $r_k\in R\cap [p_{(k-1)N+1},p_{kN}]$, $s_k\in(Q\setminus R)\cap[p_{(k-1)N+1},p_{kN}]$ such that $r_k<s_k$ for all $k\in\mathbb{Z}^+$. Introducing
$$
q_k=\log(\frac{r_k}{r_k-1})-\log(\frac{s_k}{s_k-1})>0
$$
and
$$
\advance\abovedisplayshortskip by-8pt
q'_k=\frac{1}{r_k}-\frac{1}{s_k},
$$
and referring to any function $f:\mathbb{Z}^+\to\{-1,1\}$ as a sign function, the proof is now an application of the following lemma.

\begin{lem}\label{q_sign}
With notation as above:
\begin{enumerate}
\item
\begin{equation}\label{Qszum}
\abovedisplayshortskip=-20pt
\sum_{k=1}^{\infty} q_k<\infty
\end{equation}
\item There exists an integer $M$ such that for any real number $s$ satisfying
$$
-\sum\limits_{k=M}^{\infty}q_k\le s\le\sum\limits_{k=M}^{\infty}q_k,
$$
there exists a sign function $f$ with
$$\sum\limits_{k=M}^{\infty}f(k)q_k=s.$$
\end{enumerate}
\end{lem}

\begin{proof}
To prove (1), first note that
$$
q_k=\frac{1}{r_k}+\frac{1}{2r_k^2}+O\Big(\frac{1}{r_k^3}\Big)-\Big(\frac{1}{s_k}+\frac{1}{2s_k^2}+O\Big(\frac{1}{s_k^3}\Big)\Big)=q'_k+q'_k\frac{r_k+s_k}{2r_ks_k}+O\Big(\frac{1}{r_k^3}\Big)=(1+o(1))q'_k
$$
as $k\to \infty$. Hence, for $k$ large enough, we have $0.5q_k<q_k'<1.5q_k$. The intervals $[x,x+x^{0.6}]$ are known to contain prime numbers for $x$ large enough (see e.g. \cite{BHP}), so if $k$ is large enough, we get
$$
s_k-r_k\le p_{(k+1)N}-p_{kN+1}\le Nr_k^{0.6}.
$$
It follows that
$$
q_k\le 2q'_k=2\frac{s_k-r_k}{r_ks_k}\le\frac{2Nr_k^{0.6}}{r_k^2}=\frac{2N}{r_k^{1.4}}\le\frac{2N}{k^{1.4}}
$$
for $k$ large enough, and so $\displaystyle\sum_{k=1}^{\infty}q_k$ is finite, indeed.

As for (2), we shall turn to a classical result about numerical series (see e.g. \cite{PSZ} p. 29, Exercise 131): if the assumptions
$$
b_n>0,\quad\sum_{n=1}^{\infty}b_n\text{ is finite}\quad\text{and}\quad b_k\le\sum_{n=k+1}^{\infty}b_n\text{ for all $n\in\mathbb{Z}^+$}
$$
are met and we have
$$
0\le t\le\sum_{n=1}^{\infty}b_n,
$$
then there is a function $g:\mathbb{Z}^+\to\{0,1\}$ such that
$$
t=\sum_{n=1}^{\infty }g(n)b_n
$$
(noting that in the excercise, the assumption that $b_{n}\le b_{n+1}$ for all $n\in \mathbb{Z}^+$ is unnecessary).

Now, if
$$
\advance\abovedisplayshortskip by-8pt
-\sum_{n=1}^{\infty}b_n\le s\le\sum_{n=1}^{\infty}b_n,
$$
then there is a function $g:\mathbb{Z}^+\to\{0,1\}$ such that
$$
\sum_{n=1}^{\infty }g(n)b_n=\frac{s}{2}+\frac{1}{2}\sum_{n=1}^{\infty}b_n.
$$
In particular, $f(n)=2g(n)-1$ defines a sign function $f$, and we have
$$
\sum_{n=1}^{\infty}f(n)b_n=s.
$$
Hence, it suffices to show that there is a positive integer $M$ such that
$$
q_n<\sum\limits_{k=n+1}^{\infty}q_k
$$
holds for $n\ge M$.

It is known that for $x$ large enough, there are at least $\frac{1}{2}\frac{x}{\log x}$ prime numbers between $x$ and $2x$, so there are at least $\frac{1}{2N}\frac{r_n}{\log r_n}$ prime numbers $r_k$ between $r_n$ and $2r_n$. For all these primes, we get
\[
q_k'=\frac{s_k-r_k}{s_kr_k}\ge\frac{2}{(2r_n)^2}
\]
and so
\[
\advance\abovedisplayshortskip by-8pt
q_k\ge \frac{1}{4r_n^2}.
\]
Hence,
\[
\sum_{r_n<r_k<2r_n}q_{k}\ge\frac{1}{2N}\frac{r_n}{\log r_n}\frac{1}{4r_n^2}=\frac{1}{8Nr_n\log r_n}
\advance\belowdisplayshortskip by8pt
\]
for $n$ large enough. It follows that
$$
q_n\le \frac{2N}{r_n^{1.4}},
$$
and since we have
$$
\frac{1}{8Nr_n\log r_n}\ge \frac{2N}{r_n^{1.4}}
\advance\belowdisplayshortskip by8pt
$$
for $n$ large enough, the proof of Lemma \ref{q_sign} is complete.
\end{proof}

We may now finish the proof of Lemma \ref{prim_particio}. As per the Wirsing--Odoni's theorem and (\ref{Rbecsles}), the limit
\[
\advance\abovedisplayskip by4pt
\lim_{x\to\infty}\bigg(\sum_{\substack{p\le x\\p\in R}}\log\Big(\frac{p}{p-1}\Big)-\sum_{\substack{p\le x}}\tau\log\Big(\frac{p}{p-1}\Big)\bigg)
\]
exists and is finite. By part (1) of Lemma \ref{q_sign}, the limit
\[
c=\lim_{x\to\infty}\bigg(\sum_{\substack{p\le x\\p\in R}}\log\Big(\frac{p}{p-1}\Big)-\sum_{\substack{p\le x}}\tau\log\Big(\frac{p}{p-1}\Big)-\sum\limits_{r_k,s_k<x}\frac{\log(\frac{r_k}{r_k-1})-\log(\frac{s_k}{s_k-1})}{2}\bigg)
\]
also exists and is finie. Set $H=\{r_k:k\in\mathbb{Z}^+\}\cup\{s_k:k\in\mathbb{Z}^+\}$. Conditions (\ref{tartalmaz1}) and (\ref{tartalmaz2}) imply that
$$
\sum\limits_{p\in R\setminus H}\log\Big(\frac{p}{p-1}\Big)\;\textrm{ and }\sum\limits_{p\in(Q\setminus R)\setminus H}\log\Big(\frac{p}{p-1}\Big)
$$
are divergent series, although $\log(\frac{p}{p-1})\to 0$ as $p\to \infty$ with $p\in R\setminus H$ or $p\in (Q\setminus R)\setminus H$.

At this point, we are to conditionally replace some elements of $R$, keeping the symmetric difference with the newly defined sets $R'$ and $R''$ finite. On the one hand, if
$$
c>\log (a)+\frac{1}{2}\sum\limits_{k=M}^{\infty}q_k,
$$
then we may drop finitely many elements from $R\setminus H$ to get $R'\subseteq R$ such that
\begin{equation}\label{hibatagos}
\advance\abovedisplayskip by4pt
\lim_{x\to\infty}\bigg(\sum_{\substack{p\le x\\p\in R'}}\log\Big(\frac{p}{p-1}\Big)-\sum_{\substack{p\le x}}\tau\log\Big(\frac{p}{p-1}\Big)-\sum\limits_{r_k, s_k<x}\frac{\log(\frac{r_k}{r_k-1})-\log(\frac{s_k}{s_k-1})}{2}\bigg)=\log(a)-d,
\end{equation}
where $-\frac{1}{2}\sum\limits_{k=M}^{\infty}q_k\le d\le \frac{1}{2}\sum\limits_{k=M}^{\infty}q_k.$ On the other hand, if
$$
\advance\abovedisplayskip by-5pt
c<\log (a)-\frac{1}{2}\sum\limits_{k=M}^{\infty}q_k,
$$
then we may add finitely many elements from $(Q\setminus R)\setminus H$ to $R$ to get $R'\supseteq R$ such that (\ref{hibatagos}) holds. This yields $R'$ that satisfy $(R'\cap H)=\{r_k:k\in\mathbb{Z}^+\}$, $|R'\bigtriangleup R|<\infty$ as well as ($\ref{hibatagos}$).

For the next step, note that by Lemma \ref{q_sign}, there exists a sign function $f$ such that $\frac{1}{2}\sum\limits_{k=M}^{\infty}f(k)q_k=d$. We introduce
\[R''=\left(R'\setminus \{r_k: f(k)=-1\}\right)\cup \{s_k: f(k)=-1\}.\]
Then, we get
\begin{equation*}
\begin{split}
&\lim_{x\to\infty}\bigg(\sum_{\substack{p\le x\\p\in R''}}\log\Big(\frac{p}{p-1}\Big)-\tau\sum_{\substack{p\le x}}\log\Big(\frac{p}{p-1}\Big)\bigg)=\\
&\lim_{x\to\infty}\bigg(\sum_{\substack{p\le x\\p\in R'}}\log\Big(\frac{p}{p-1}\Big)-\sum_{r_k<x,\;f(k)=-1}\log\Big(\frac{r_k}{r_k-1}\Big)-\log\Big(\frac{s_k}{s_k-1}\Big)-\tau\sum_{\substack{p\le x}}\log\Big(\frac{p}{p-1}\Big)\bigg)=\\
&\lim_{x\to\infty}\bigg(\sum_{\substack{p\le x\\p\in R'}}\log\Big(\frac{p}{p-1}\Big)-\tau\sum_{\substack{p\le x}}\log\Big(\frac{p}{p-1}\Big)-\sum_{r_k<x}\frac{\log(\frac{r_k}{r_k-1})-\log(\frac{s_k}{s_k-1})}{2}-\\
&\mkern23mu\sum_{r_k<x,\;f(k)=-1}\frac{\log(\frac{r_k}{r_k-1})-\log(\frac{s_k}{s_k-1})}{2}+\sum_{r_k<x,\;f(k)=1}\frac{\log(\frac{r_k}{r_k-1})-\log(\frac{s_k}{s_k-1})}{2}\bigg)=\\
&\lim_{x\to\infty}\bigg(\sum_{\substack{p\le x\\p\in R'}}\log\Big(\frac{p}{p-1}\Big)-\tau\sum_{\substack{p\le x}}\log\Big(\frac{p}{p-1}\Big)-\\
&\mkern51mu\sum_{r_k<x}\frac{\log(\frac{r_k}{r_k-1})-\log(\frac{s_k}{s_k-1})}{2}+\sum_{r_k<x}f(k)\frac{\log(\frac{r_k}{r_k-1})-\log(\frac{s_k}{s_k-1})}{2}\bigg)=\log(a)-d+d=\log(a).
\end{split}
\end{equation*}
In particular,
\[
\lim_{x\to\infty}\bigg(\prod_{\substack{p\le x\\p\in R''}}\frac{p}{p-1}\bigg)\bigg(\prod_{\substack{p\le x}}\frac{p-1}{p}\bigg)^{\tau}=a.
\]
By construction, there exists $N_1$ such that $|R(x)-R'(x)|\le N_1$, and we have $|R'(x)-R''(x)|\le 1$ for all $x\in\mathbb{Z}^+\!$, so
\begin{equation}\label{RRbecsles}
|R''(x)-R(x)|\le N'+1.
\end{equation}
Now, by (\ref{Rbecsles}), we have
$$
R''(x)=\tau \frac{x}{\log x}+O(\frac{x}{\log ^2x}).
\advance\belowdisplayshortskip by8pt
$$
Also, by (\ref{tartalmaz1}), we have $|R\cap\{p_n,\ldots,p_{n+(N'+2)N-1}\}|\ge 2N'+3$ for all $n\in\mathbb{Z}^+\!$. By (\ref{RRbecsles}), we thus get
$$
|R''\cap \{p_n,\ldots,p_{n+(N'+2)N-1}\}|\ge 1,
$$
which, then, completes the proof of Lemma \ref{prim_particio}.
\end{proof}

With that, we have demonstrated Theorem \ref{precision:bound:2}.
\end{proof}

\vskip8pt
\begin{proof}[Proof of theorem \ref{liminf:max}]
The proof of part (1) is an application of Theorem \ref{PS}. If $(A_1,\ldots,A_h)\in\MC_h$, then $A_1\cup\cdots\cup A_h\in\MB_h$ and
    \[(A_1\cup A_2\cup \ldots \cup A_h)(x)\le h\max\{A_1(x),\dots ,A_h(x)\}.\]
It now follows from Theorem \ref{PS} that
\[
\liminf_{x\to\infty}\frac{(A_1\cup\cdots\cup A_h)(x)}{\frac{x}{\log x}}>1,
\]
and so
\[
\liminf_{x\to\infty}\frac{\max\{A_1(x),\ldots,A_h(x)\}}{\frac{x}{\log x}}>\frac{1}{h},\]
which was to be shown.

We now prove part (2). The construction is based on the construction given by Pach and S\' andor in their proof of Theorem \ref{PS}. Set $[n]=\{1,\ldots,n\}$ for $n\in\mathbb{Z}^+$. We shall construct a strictly increasing sequence $\seq{n_i}{i}\subseteq\mathbb{Z}^+$ and sets $A_1^i,A_2^i,\ldots A_h^i\subset [n_i]$ such that the following conditions hold for all $i\ge1$:
\begin{enumerate}
	\item\label{maxmin:1} $A_1^i,A_2^i,\ldots A_h^i$ are multiplicative complements for $[n_i]$;
	\item\label{maxmin:2} $\max\{|A_1^i|,\ldots,|A_h^i|\}\le\left(\frac1h+\varepsilon\right)\frac{n_i}{\log n_i}$;
	\item\label{maxmin:3} $A_j^i\cap[n_{i-1}] = A_j^{i-1}\:\text{for all }1\le j\le h$.
\end{enumerate}
The sets $A_j=\bigcup_i A_j^i$ for $1\le j\le h$ then constitute multiplicative complements of order $h$. Moreover, condition \ref{maxmin:3} ensures that $A_j\cap [n_i] = A_j^i$, and by condition \ref{maxmin:2}, we get
\[
\max\{A_1(n_i),\ldots,A_h(n_i)\}\le\Big(\frac1h+\varepsilon\Big)\frac{n_i}{\log n_i}
\]
and hence
\[
\liminf_{x\to\infty}\max\{A_1(x),\ldots,A_h(x)\}\frac{\log x}{x}\le\frac{1}{h}+\varepsilon.
\advance\belowdisplayshortskip by4pt
\]
Let us set $n_0=N=\left\lceil\frac{256}{\varepsilon^2}\right\rceil+1$ and $A_1^0=A_2^0=\ldots=A_h^0=[n_0]$. We define the sequence $\seq{n_i}{i}$ and the sets $A_1^i,A_2^i,\ldots A_h^i$ recursively. Setting $x=n_i$, we shall pick $y=n_{i+1}$ large enough. Conditions on $y$ are imposed as we proceed with the proof. The first two conditions are $x<\frac{\varepsilon}{5}\mkern2mu y$ and $x^2\!<y$. We define{\parfillskip=0pt\par}%
\vskip-26pt\hskip12pt
\begin{minipage}[t]{.40\linewidth}
$$
A_j^{i+1}=A_j^i\cup F_0\cup F_1\cup F_2\cup F_3^j,\text{ where}
$$
\end{minipage}%
\begin{minipage}[t]{.5\linewidth}
\begin{align*}
F_0&=\left\{i:x<i<x y^{2/3}\right\}\!,\\
F_1&=\left\{pv:y^{2/3}<p<\frac{y}{x},\;p\text{ is prime},\,v\le\sqrt{x}\right\}\!,\\
F_2&=\left\{pv:\frac{y}{x}<p\le\frac{y}{N},\;p\text{ is prime},\,v\le\sqrt{\frac{y}{p}}\right\}\!,\\
F_3^j&=\left\{p:\frac{y}{N}<p\le y,\;p\text{ is prime}\colon \pi(p)\equiv j\;({\rm mod}\;h)\right\}\!.
\end{align*}
\end{minipage}
\vskip10pt
Note that if $y>x^2$, then we have $\min(y^{2/3},y/x)>x$, so for all $1\le j\le h$, the elements of $A_j^{i+1}\setminus A_j^i$ are larger than $x$. In particular, condition (\ref{maxmin:3}) is met.

We now use induction to show that every number $m:\;1\le m\le y$ can be written in the form $m=st$ with $s\in A_j^{i+1}$ and $t\in A_k^{i+1}$ for some $1\le j\ne k\le h$. Having this representation implies at once that $A_1^{i+1},A_2^{i+1},\ldots,A_h^{i+1}$ are multiplicative complements for $[y]$ since $1\in A_j$ for all $1\le j\le h$. Pick hence $m\leq y$. It is easy to see that $m$ can be written as $m=uv$ with $v\leq u$ such that either $u\leq y^{2/3}$ or $u>y^{2/3}$ is a prime number (e.g., see \cite{erdos1938}).

First, assume that $u\leq y^{2/3}$ is a prime number. If $x<v$, then both $u$ and $v$ lie in $(x,y^{2/3}x]$, so  we have $u,v\in A_1^{i+1}\cap A_2^{i+1}$, and we may write $m=uv$ with $u\in A_1^{i+1}$ and $v\in A_2^{i+1}$. If $v\leq x$, then we distinguish two cases:
\begin{itemize}
\item If $x<uv$, then $m$ can be written as $m=(uv)\cdot 1$ since $uv$ lies in $(x,y^{2/3}x]$.
\item If $uv\leq x$, then $m=uv$ can be written as $m=st$ with $s\in A_j^{i}$, $t\in A_k^{i}$ for some $1\le j\ne k\le h$ by the induction hypothesis.
\end{itemize}

Assume now $u>y^{2/3}$ to be some prime $p$.
\begin{itemize}
\item If $y^{2/3}<p\leq y/x$, then $p\in F_1\subset A_1^{i+1}\cap A_2^{i+1}$, and as $m\leq y$, we have $v=m/p\leq y/p\leq y^{1/3}$. If $x<v$, then $v\in F_0\subset A_1^{i+1}\cap A_2^{i+1}$, so $m$ can be written as $m=p\mkern1mu v$ with $p\in A_1^{i+1}$, $v\in A_2^{i+1}$. If $v\leq x$, then by the induction hypothesis, $v$ can be written as $v=v_1v_2$ with $v_1\in A_j^i$, $v_2\in A_k^i$ for some $1\le j\ne k\le h$. Without loss of generality, we may assume that $v_1\leq v_2$. Then, since $v_1\leq \sqrt{v_1v_2}=\sqrt{v}\leq\!\sqrt{x}$, we have $p\mkern1mu v_1\in F_1\subset A_j^{i+1}\cap A_k^{i+1}$ and $v_2\in A_k^i\subset A_k^{i+1}$, and hence $m$ can be written as $m=(pv_1)v_2$ with $p\mkern1mu v_1\in A_j^{i+1}$, $v_2\in A_k^{i+1}$.
\item If $y/x<p$ and $y/N \leq p$, then $p\in F_3^1\cup F_3^2\cup\ldots \cup F_3^h$, so there exists $1\le j\le h$ such that $p\in A_j^{i+1}$ and $v=m/p\leq y/p\leq N$. Consequently, there exists $1\le k\le h$, $k\ne j$ such that $v\in A_k^{i+1}$ since $[N]\subseteq A_1^{i+1}\cap\cdots\cap A_h^{i+1}$. In particular, $m$ can be written as $m=p\mkern1mu v$.
\item If $y/x<p<y/N$, then there exists $j$  with $N\leq j\leq x-1$ and $y/(j+1)<p\leq y/j$, so we have $v\le \frac{y}{p}< x$. By the induction hypothesis, we can find $v_1\in A_j^i$, $v_2\in A_k^i$ for $1\le j\ne k\le h$ such that $v=v_1v_2$. As before, without loss of generality, we may assume that $v_1\leq v_2$, so we have $v_1\leq \sqrt{v}\leq \sqrt{y/p}$. Then, $pv_1\in F_2\subset (A_j^{i+1}\cap A_k^{i+1})$ and $v_2\in A_k^i\subset A_k^{i+1}$, and hence $m$ can be written as $m=(pv_1)v_2$ with $pv_1\in A_j^{i+1}$, $v_2\in A_k^{i+1}$.
\end{itemize}
This shows that condition (\ref{maxmin:1}) is also met.

We now prove that for all $1\le j\le h$, condition (\ref{maxmin:2}) is met for $A_j^{i+1}$ and $y=n_{i+1}$. If $x^4<y$, then
$$
|A_j^i\cup F_0|=|A_j^i\cup\{i:x<i\leq y^{2/3}x\}|\leq y^{2/3}x<y^{11/12}<\frac{\varepsilon}{5}\,\frac{y}{\log y}
$$
for $y$ large enough. Moreover, if $x^4<y$ and $x\geq N>\frac{256}{\varepsilon ^2}$, then
\begin{align*}
|F_1|&=|\{pv:y^{2/3}<p\leq y/x,\;p\text{ is a prime},\;v<\sqrt{x}\}|\\
&\leq\pi(y/x)\sqrt{x}\leq 2\frac{y/x}{\log(y/x)}\sqrt{x}=2\frac{y}{\log y}\frac{1}{\sqrt{x}}\frac{1}{1-\frac{\log x}{\log y}}\leq\frac{\varepsilon}{5}\,\frac{y}{\log y}.
\end{align*}
If $y>x^{200/\varepsilon}$, then
\begin{align*}
|F_2|&=|\{pv:y/x<p\leq y/N,\;p\text{ is a prime},\;v\leq\sqrt{y/p}\}|\\
&\leq|\{pv:\exists j:N\leq j\leq x-1,\;y/(j+1)<p\leq y/j,\;p\text{ is a prime},\;v\leq\sqrt{j+1}\}|\\
&\leq\sum\limits_{j=N}^{x-1}\bigg(\pi\Big(\frac{y}{j}\Big)-\pi\Big(\frac{y}{j+1}\Big)\bigg)\sqrt{j+1}.
\end{align*}
Note that with $x$ fixed and as $y\to\infty$, we have $\pi\big(\frac{y}{j}\big)=\frac{y}{j\log y}+o\big(\frac{y}{j(\log y)^{1.5}}\big)$, yielding
$$
\pi\Big(\frac{y}{j}\Big)-\pi\Big(\frac{y}{j+1}\Big)=\frac{1}{j(j+1)}\frac{y}{\log y}+o\Big(\frac{y}{j(\log y)^{1.5}}\Big)
\advance\belowdisplayskip by8pt
$$
(for instance, it suffices to take $y=x^x$). Hence, we may write
$$
\sum\limits_{j=N}^{x-1}\bigg(\pi\Big(\frac{y}{j}\Big)-\pi\Big(\frac{y}{j+1}\Big)\bigg)\sqrt{j+1}=\bigg(\sum\limits_{j=N}^{x-1}\frac{1}{j\sqrt{j+1}}\bigg)\frac{y}{\log y}+o\Big(\frac{\sqrt{x}}{\sqrt{\log y}}\Big)\frac{y}{\log y},
$$
where
$$
\advance\abovedisplayshortskip by-8pt
\sum\limits_{j=N}^{x-1}\frac{1}{j\sqrt{j+1}}<\int_{\frac{256}{\varepsilon ^2}}^{\infty}\frac{1}{x^{3/2}}dx=\frac{\varepsilon }{8}.
$$
In particular, we have
$$
|\{pv:y/x<p\leq y/N,\;p\text{ is a prime},\;v\leq\sqrt{ y/p}\}|\leq\frac{\varepsilon}{5}\,\frac{y}{\log y}
$$
for $y$ large enough relative to $x$. This accounts now for the last term in the decomposition of $A_j^{i+1}$ as
$$
\advance\abovedisplayskip by-8pt
|F_3^j|=|\{p:y/N<p\leq y,\;p\text{ is a prime},\;\pi(p)\equiv j\;{\rm mod}\;h\}|\leq 1+\frac{\pi(y)}{h}\leq\left(\frac{1}{h}+\frac{\varepsilon}{5}\right)\frac{y}{\log y}
$$
for $y$ sufficiently large. Putting this all together, for all $1\le j\le h$ and for $y$ sufficiently large, we get
$$
|A_j^{i+1}|\leq |A_j^i|+|F_0|+|F_1|+|F_2|+|F_3^j|\le\Big(\frac{1}{h}+\varepsilon\Big)\frac{y}{\log y},
$$
completing the proof.
\end{proof}

\bigskip

\renewcommand{\refname}{Bibliography}

\end{document}